\documentclass[12pt,twoside, final]{amsart}

\usepackage{amsmath,amsthm,amscd,amsfonts,amssymb,enumerate}
\usepackage{graphicx}
\usepackage{color}
\usepackage[colorlinks]{hyperref}
\usepackage{amsfonts,amssymb,amscd,amsmath,enumerate,url,verbatim}
 \usepackage[dvips]{epsfig}
 \usepackage[none]{hyphenat}
\usepackage{amsmath,amssymb,amsfonts,enumerate,amsthm}
 \usepackage{amsgen, amstext,amsbsy,amsopn, amsthm, amsfonts,amssymb,amscd,amsmath,euscript,enumerate,url,verbatim,calc,xypic}
 \usepackage{latexsym}
 \usepackage{graphics}
 \usepackage{color}
\numberwithin{equation}{section}
\newtheorem{thm}{Theorem}[section]
\newtheorem{lem}[thm]{Lemma}
\newtheorem{cor}[thm]{Corollary}

\newtheorem{rem}[thm]{Remark}

\newtheorem{nota rem}[thm]{Notation and Remark}

\newtheorem{defen rem}[thm]{Definition and Remark}

\def\hei{\operatorname{ht}}

\def\pd{\operatorname{pd} }
 \def\Ext{\operatorname{Ext}}
\def\Hom{\operatorname{Hom}}
\def\Spec{\operatorname{Spec}}

\def\grade{\operatorname{grade}}
 
\def\cd{\operatorname{cd}}

\def\Supp{\operatorname{Supp}}

\def\Ass{\operatorname{Ass}}
 \def\dim{\operatorname{dim}}
\newcommand{\p}{\mathfrak{p}}

\newcommand{\fa}{\mathfrak{a}}

\newcommand{\m}{\mathfrak{m}}

\newcommand{\Z}{\mathbb{Z}}

\newcommand{\NN}{\mathbb{N}}

\newcommand{\Q}{\mathbb{Q}}

\begin{document}
\bibliographystyle{amsplain}

\author{Maryam Jahangiri}
\address{ Department of Mathematics, Faculty of Mathematical Sciences and Computer, Kharazmi University, Tehran,
Iran.}
\email{ jahangiri@khu.ac.ir, jahangiri.maryam@gmail.com }
 

\title[Bass numbers of graded components of... ]
{Bass numbers of graded components of local cohomology modules}

\begin{abstract}
Let $R=\bigoplus_{n\in \NN_0}R_n$ be a standard graded ring, $R_+=\bigoplus_{n\in \NN}R_n$ its irrelevant ideal,   and $M$   a finitely generated graded $R$-module.  In this paper, we study the asymptotic behavior of the sequence
$\{\mu^i(\p_0, H^j_{R_+}(M)_n)\}_{n\in \Z}$ of Bass numbers of graded components of local cohomology modules with respect to an ideal $\p_0\in \Spec(R_0)$  in each of the following cases:
\begin{enumerate}
   \item $i=0$ or $i= 1$ and $j\leq f_{R_+}(M)$,
     \item  $R_0$ is regular, 
   $i=  \hei(\p_0)$ or $i=  \hei(\p_0)- 1$ and  $j= \cd_{R_+}(M)$,
   \item $M$ is relative Cohen-Macaulay with respect to $R_+$.
     \end{enumerate}
  Here, $\cd_{R_+}(M)$ and $f_{R_+}(M)$ denote the cohomological dimension and finiteness dimension of $M$ with respect to $R_+$, respectively.
\end{abstract}
\maketitle
\section{introduction}

Throughout this paper, all rings are assumed to be commutative Noetherian with identity. Let $R= \bigoplus_{n\in \NN_0} R_n$ be  a standard graded ring, i.e., $R_0$ is a commutative Noetherian ring with identity, and $R$ is generated as an $R_0$-algebra by finitely many elements of degree 1. Let $R_+= \bigoplus_{n\in \NN} R_n$ denote the irrelevant ideal of $R$ and $M= \bigoplus_{n\in \Z} M_n$  be a finitely generated graded $R$-module. We denote by $\Q$, $\Z$, $\NN_0$ and $\NN$ the sets of rational numbers, integers, non-negative integers, and positive integers, respectively.

Let $S$ be a ring and $\fa$ an ideal of $S$. For any $i\in \NN_0$, $H^i_{\fa}(-)$ denote the $i$-\it{th local cohomology functor with respect to} $\fa$ which is the $i$-th right derived functor of the $\fa$-torsion functor $\Gamma_{\fa}(-)$. For any $S$-module $X$,
 \[\Gamma_{\fa}(X):= \cup_{t\in X} (0:_X \fa^t)= \{x\in X\mid \exists t\in \NN, \fa^t x=0\}.\]
 Reference \cite{bsh} provides a comprehensive treatment of the basic properties of local cohomology modules.

The vanishing and finiteness properties of local cohomology modules are central and highly interesting problems in this theory; see for example \cite{al, b, ckhs, dnt} and \cite{l}.

It is well-known that if $X =\bigoplus_{n\in \Z}X_n$ is a graded $R$-module  and $\fa$ is a homogeneous ideal of $R$, then for any $i\in \NN_0$, $H^i_{\fa}(X)$ naturally inherits a grading. In particular (\cite[Theorem 16. 1. 5]{bsh}):
 \begin{enumerate}
 
      \item  $H^i_{R_+}(M)_n$ is a finitely generated  $R_0$-module for any $n\in \Z$, and
     \item $H^i_{R_+}(M)_n= 0$ for all $n\gg 0$; that is, there exists $k\in \Z$ such that $H^i_{R_+}(M)_n= 0$ for all $n\geq k.$
\end{enumerate}
  In \cite{in} and \cite{jz}, the above results have been studied for ideals that either contain or are contained in the irrelevant ideal.

 Although, the graded components of local cohomology modules $H^i_{R_+}(M)_n$ are trivial in almost all positive degrees, they can behave quite irregularly in negative degrees. For example in \cite[Theorem 3. 3]{ckhs} the authors show that if $p\geq 11$ is a prime number with $p\equiv 2$ (mod 3) and $K$ is a field of characteristic $p$, then there exists a normal standard graded $K$-algebra $R$  such that for any $n\in \NN,$
  \[
\dim_K(H^2_{R_+}(R)_{-n})=
\begin{cases}
    1 & \text{if } n\equiv 0 \,(mod \, p+ 1), \\
    1 & \text{if } n= p^t \text{for some odd integer} \, t, \\
    0 & \text{otherwise.}
\end{cases}
\]

Nevertheless, it is possible to predict the asymptotic behavior of the graded components
 $H^i_{R_+}(M)_n$ as $n\rightarrow-\infty$ in certain cases, for example:
\begin{enumerate}

    \item $i= \cd_{R_+}(M)$, the cohomological dimension  of $M$ with respect to $R_+$(\cite{b}) (see (\ref{defcd}) for definition of cohomological dimension.)
    \item $i= f_{R_+}(M)$, the finiteness  dimension  of $M$ with respect to $R_+$(\cite{b}) (see (\ref{deff}) for definition of finiteness dimension.)
    \item When $R_0$ is of small dimension (\cite{b}).
    \item When $R_0$ is a regular ring and $M= R$ is a polynomial ring over $R_0$; in this case, one may consider any homogeneous ideal instead of $R_+$ (\cite{tony36}).
\end{enumerate}

In this paper we consider the asymptotic behavior of the sequence $\{\mu^i(\p_0, H^j_{R_+}(M)_n)\}_{n\in \Z}$ of Bass numbers of  the graded components of local cohomology modules with respect to an ideal $\p_0\in \Spec(R_0)$ (see \ref{Bass} for definition of Bass numbers).
Bass numbers of  local cohomology modules $H^i_{\fa}(N)$, where $\fa$ is an ideal of a ring $S$ and $N$ is an $S$-module,  have been investigated in several special cases. For example:
\begin{enumerate}

    \item $N= S$ is a regular local ring containing a field (\cite{hu} and \cite{lyu});
    \item $S$ is a semigroup ring (\cite{helm});
    \item $N$ is a finitely generated $S$-module, $i\in \{ 0, 1, 2\}$ or $R$ is regular and $i\in\{ \hei(\p),  \hei(\p)- 1\}$ (\cite{j1}); 
    \item $N= S$ is a polynomial ring over a field and $\fa$ is a square-free monomial ideal, (\cite{al} and \cite{yan}.)
\end{enumerate}

This paper is devoted in to 4 sections.
In Section 2, we present some preliminaries, with a primary focus on the graded components of *injective modules.
We show that if $R$ is flat over $R_0$ via $R_0\hookrightarrow R$ and $X= \bigoplus_{n\in \Z} X_n$ is a graded  $R$-module then, for any $n\in \Z$, $i\in \NN_0$ and any ideal $I_0$ of $R_0$ we have
\[ ^*Ext^i_{R}(R/I_0R, X)_n\cong Ext^i_{R_0}(R_0/I_0, X_n).\]

In Section 3,  we study the Bass numbers of graded components of local cohomology modules $\mu^i(\p_0, H^j_{R_+}(M)_n)$ where $j\leq f_{R_+}(M)$ (see \ref{deff} for definition of 
$f_{R_+}(M)$).
 To be more precise, it is shown (Theorem \ref{thm1}) that
if $(R_0,\m_0)$ is a local ring then for any $j\leq f_{R_+}(M)$ we have:
 \begin{enumerate}
\item 
 $\mu^{0}(\m_0, H^j_{R_+}(M)_n)=   \dim_{\tfrac{R_0}{\m_0}}(H^{j }_{R_+}(\tfrac{R}{\m_0R}, M)_n), \,\,\,\text{ for all}\,\, n\ll0.$
\item   $\mu^1(\m_0, H^j_{R_+}(M)_n)\leq \dim_{\tfrac{R_0}{\m_0}}(H^{j+ 1}_{R_+}(\tfrac{R}{\m_0R}, M)_n), \,\,\,\text{ for all}\,\, n\ll0.$
    \end{enumerate}

In Section 4,  we consider the Bass numbers   $\mu^i(\p_0, H^j_{R_+}(M)_n)$ where $j= \cd_{R_+}(M)$ (see \ref{defcd} for definition of $\cd_{R_+}(M)$).
More precisely, among other results, we show  (Theorem \ref{thm}) that if $(R_0, \m_0)$ is a regular local ring of dimension $d:= \dim(R_0)$ and $c:= \cd_{R_+}(M)$ then 
the following statements hold.
\begin{enumerate}
    
\item There exists a polynomial $p(x)$ of degree $< c$ such that
\[\mu^d(\m_0, H^c_{R_+}(M)_n)= p(n) \,\,\, \text{for  all} \,\, n\ll0.\]
In particular, if $c= 1$, then the sequence $\{\mu^d(\m_0,
H^1_{R_+}(M)_n)\}_{n\in \Z}$ of integers is ultimately constant; that is, there exists $\alpha\in \Z$  such that
 \[ \mu^d(\m_0, H^1_{R_+}(M)_n)=\alpha, \,\,\,\,\,\,\,\,\   \text{for  all}\,\,n\ll0.\]

\item If $c> 0$ and   $\mu^d(\m_0, H^c_{R_+}(M))\neq 0$, then
$\mu^d(\m_0, M)\neq 0$ and the sequence $\{\mu^d(\m_0, H^c_{R_+}(M)_n)\}_{n\in \Z}$ is asymptotically decreasing. Moreover,  if $c\geq 2$ it is asymptotically strictly decreasing.

\item If $H^c_{R_+}(\Ext^{d- 1}_{R }(\tfrac{R}{\m_0R}, M))$ is finitely generated then there exists a polynomial $p(x)\in \Q[x]$ of degree $< c-1$ such that
\[ \mu^{d- 1}(\m_0, H^c_{R_+}(M)_n)\leq  p(n), \,\,\, \text{for   all}\,\, n\ll0.\]
If, in addition, $\Ext^{d}_{R }(\tfrac{R}{\m_0R}, H^{c- 1}_{R_+}(M))$ is finitely generated then the equality holds.
\end{enumerate}
 
We also study the Bass numbers $\mu^{i}(\m_0, H^j_{R_+}(M)_n)$ in the case where $M$ is relative Cohen-Macaulay with respect to $R_+$ (see \ref{defrcm} for definition of relative Cohen-Macaulay modules). In particular, we show  (Corollar \ref{rcm}) that 
  if $(R_0, \m_0)$ is local,  $M$ is relative Cohen-Macaulay with respect to $R_+$ of degree $c$ and $R$ is flat as an $R_0$-module, then   the following statements holds.
\begin{enumerate}
    \item For any $i\in \NN_0$ and  $n\in \Z$, $\mu^i(\m_0, H^c_{R_+}(M)_n)=  \dim_{\tfrac{R_0}{\m_0}}(H^{i+ c }_{R_+}(\tfrac{R}{\m_0R}, M)_n).$
 \item Let  $(R_0, \m_0)$ be regular of dimension $d$. Then,
 \begin{enumerate}
     \item there exists a polynomial $p(x)\in \Q[x]$ of degree $< c- 1$ such that 
     \[\mu^{d- 1}(\m_0, H^c_{R_+}(M)_n)\geq p(n),\,\,\,\,\,\,\,\,\text{for all}\,\,n\ll 0,\]
     \item if $H^{c- 1}_{R_+}(\Ext^d_R(\tfrac{R}{\m_0R}, M))$ is not finitely generated, then 
     \[\mu^{d- 1}(\m_0, H^c_{R_+}(M)_n)> 0,\,\,\,\,\,\,\,\,\text{for all}\,\,n\ll 0.\]
     Therefore, $\mu^{d- 1}(\m_0, H^c_{R_+}(M))$ is not finite.
 \end{enumerate}
\end{enumerate}
 

Here, the notation $n\ll0$ ($n\gg 0$),   means  "for sufficiently small (large) values of $n$".
We maintain the assumptions and settings introduced in the Introduction throughout the paper.
\section{Preliminaries}
Throughout the paper,  $ ^*\mathfrak{C}_R$ and $ ^*\mathfrak{C}^f_R$  denote, respectively, the category of  graded and finitely generated graded $R$-modules and  homogeneous homomorphisms of degree zero. The injective objects in $ ^*\mathfrak{C}_R$ 
 are called   *injective modules. For any graded $R$-module $X$, we use   $ ^*E_R(X)$ to denote the *injective envelope of $X$. Note that, in view of \cite[Theorem 13. 2. 4]{bsh}, $ ^*\mathfrak{C}_R$ has enough injectives.
 
 The functor $ ^*\Hom_R(-, -)$
is the restriction of the usual $\Hom$-functor to  $ ^*\mathfrak{C}_R$ and for any $i\in \NN_0$, $ ^*\Ext^i_R(-, -)$ denotes its $i$-th right derived functor in  $ ^*\mathfrak{C}_R$.

\begin{rem}\label{rem}
The following statements will be used several times in the paper. We present them here for the reader's convenience.
  \begin{enumerate}
      \item (\cite[13. 1. 8(iv)]{bsh})\label{bh1.5.19} Let $M$ and $N$ be two graded $R$-modules with $M$ finitely generated. Then for any $i\in \NN_0:$\[ ^*\Ext^i_R(M, N)= \Ext^i_R(M, N),\] (here we do not need $R$ to be standard graded).

\item (\cite[Theorem 10. 74]{r})\label{rot10.74} Assume that $\Phi: S\longrightarrow T$ is a flat ring homomorphism, $X$ is an
$S$-module and $Y$ is a $T$-module. Then
for all $i\in \NN_0$:
\[\Ext^i_T(X\otimes_S T, Y)\cong \Ext^i_S(X, Y).\]
\item (\cite[24. 10]{wis})\label{limm} Let $S$ be a ring and $X$  an $S$-module. Then  $X$ is finitely generated if and only if the natural homomorphism
\[\phi_X : \underset{ \overset {\longrightarrow}{i\in \Lambda}}{ \lim } \Hom_S(X, M_i )\longrightarrow \Hom_S(X, \underset{ \overset {\longrightarrow}{i\in \Lambda}}{ \lim }M_i )\]
 is an isomorphism for every direct system $(M_i, f_{ij})_{i, j\in \Lambda}$ of modules in $\sigma(X)$
with $f_{ij}$ monomorphisms.
  \end{enumerate}
\end{rem}

The following  lemma studies the graded components of an *injective $R$-module.
\begin{lem}\label{1}
Assume that $R$ is a flat $R_0$-module via $R_0\hookrightarrow R$ and let $E= \bigoplus_{n\in \Z} E_n$ be an *injective $R$-module. Then, $E_n$ is an injective $R_0$-module for all $n\in \Z$.
\end{lem}
\begin{proof}
Let $I_0$ be an ideal of $R_0$. Then, the result follows using  the following chain of isomorphisms, for any $i\in \NN$.
\begin{align}
     0&=  ^*\Ext^i_{R}(R/I_0R, E) \nonumber\\
& \cong \Ext^i_{R}(R/I_0R, E) && \text{ (by   \ref{rem}(\ref{bh1.5.19}))}\nonumber \\
&\cong \Ext^i_{R_0}(R_0/I_0, E)   &&\text{(by  \ref{rem}(\ref{rot10.74}))}\nonumber \\
&\cong \bigoplus_{n\in \Z} \Ext^i_{R_0}(R_0/I_0, E_n). && \text{(by  \ref{rem}(\ref{limm}))}\nonumber
\end{align}

\end{proof}
The above lemma implies that if $X= \bigoplus_{n\in \Z}X_n$ is a graded $R$-module and
$$ ^* E^{\bullet}_R(X): 0\rightarrow E^0\rightarrow \cdots \rightarrow E^i\stackrel{d^i}{\rightarrow} E^{i+ 1}\rightarrow \cdots  $$
 denotes an *injective resolution of $X$, where for any $i\in \NN_0$, $E^i= \bigoplus_{n\in \Z} E^i_n$ is the $i$-th *injective module in $ ^*E^{\bullet}_R(X)$ and
 $d^i=(d^i_n)_{n\in \Z}: E^i\rightarrow   E^{i+ 1}$, then for any $n\in \Z$,
 $$ ^*E^{\bullet}_R(X)_n: 0\rightarrow E^0_n\rightarrow \cdots \rightarrow E^i_n\stackrel{d^i_n}{\rightarrow} E^{i+ 1}_n\rightarrow \cdots $$
 is an injective resolution of $X_n$ as an $R_0$-module, in the case where $R$ is   flat over $R_0$  via $R_0\hookrightarrow R$.

\begin{lem}\label{2}
Assume that $R$ is a flat $R_0$-module via $R_0\hookrightarrow R$ and let $X= \bigoplus_{n\in \Z} X_n$ be a graded  $R$-module. Then, for any $n\in \Z$, $i\in \NN_0$ and any ideal $I_0$ of $R_0$ we have \[ \Ext^i_{R}(R/I_0R, X)_n\cong \Ext^i_{R_0}(R_0/I_0, X_n).\]
\end{lem}
\begin{proof}
Assume that $^*E^{\bullet}_R(X)$ denote an *injective resolution of $X$. Then, for any $n\in \Z$, $i\in \NN_0$ and any ideal $I_0$ of $R_0$ we have
 \begin{align*}
  \Ext^i_{R}(R/I_0R, X)_n&=    ^*\Ext^i_{R}(R/I_0R, X)_n   && \text{(by \ref{rem}(\ref{bh1.5.19}))}\\
  &=H^i( ^*\Hom_R(R/I_0R,  ^*E^{\bullet}_R(X)))_n  \\
  &= H^i(\Hom_R(R/I_0R,  ^*E^{\bullet}_R(X))_n)   &&\text{ (by  \ref{rem}(\ref{bh1.5.19}))} \\
  &\cong  H^i(\Hom_{R_0}(R_0/I_0,  ^*E^{\bullet}_R(X))_n)    &&\text{ (by  \ref{rem}(\ref{rot10.74})) }
  \\
 &=  H^i(\Hom_{R_0}(R_0/I_0, ^*E^{\bullet}_R(X)_n))   \\
 &\cong  H^i(\Hom_{R_0}(R_0/I_0, E^{\bullet}_{R_0}(X_n)))    &&\text{ (by   \ref{1})}   \\
 &=  \Ext^i_{R_0}(R_0/I_0, X_n).
  \end{align*}
\end{proof}
In the rest of this section, we  require the notion of generalized local cohomology modules. This concept was introduced by Jürgen Herzog in \cite{h}. 

Let $\fa$ be an ideal of a ring $S$ and let  $X$ and $Y$ be $S$-modules. For each $i\in \NN_0$, the $i$-th   generalized local cohomology module of $X$
and $Y$ with respect to $\fa$ is defined by
\[H^i_{\fa}(X, Y):= \underset{ \overset {\longrightarrow}{t\in \mathbb{N}}}{ \lim }\Ext^i_{\fa}(
\tfrac{X}{\fa^t X}, Y).\]
Note that if $X= S$ then, in view of \cite[Theorem 1. 3. 8]{bsh}, this concept coincides with the usual local cohomology modules of $Y$ with respect to $\fa$; that is, $H^i_{\fa}(S, Y)\cong H^i_{\fa}(Y)$.

The following version of the Grothendieck spectral sequence will be used in the subsequent lemma.
\begin{lem} (\cite[Theorem 10. 47]{r}) \label{rot10.47}
      Let $\mathcal{A}\stackrel{\mathcal{G}}{ \rightarrow} \mathcal{B} \stackrel{\mathcal{F}}{ \rightarrow}\mathcal{C}$
  be covariant additive functors, where $\mathcal{A}$, $\mathcal{B}$, and $\mathcal{C}$ are abelian categories with enough injectives. Assume that $\mathcal{F}$ is left exact and that $\mathcal{G}(E)$ is right $\mathcal{F}$-acyclic for every
injective object $E$ in $\mathcal{A}$. Then, for every object $A$ in $\mathcal{A}$, there is a third quadrant spectral sequences with
\[E_2^{i, j}= R^i\mathcal{F}(R^j\mathcal{G}(A))\underset{i}{\Rightarrow}R^{i+ j}\mathcal{F}\mathcal{G} (A).\]
\end{lem}

The following lemma plays a fundamental role in establishing the main results of this paper.

\begin{lem}\label{spec}
Let $X$ be a finitely generated graded $R$-module. Then there exist  the following homogeneous convergences of spectral sequences.
\begin{enumerate}
    \item \label{spec1}  $E_2^{i, j}= \Ext^i_{R }(X, H^j_{R_+}(M))\underset{i}{\Rightarrow}H^{i+ j}_{R_+}(X, M).$
    \item \label{spec2} $E_2^{i, j}= H^i_{R_+}(\Ext^j_{R }(X, M))\underset{i}{\Rightarrow}H^{i+ j}_{R_+}(X, M).$
\end{enumerate}
 
\end{lem}
\begin{proof}
\begin{enumerate}
    \item 
Let $E$ be an *injective $R$-module. Then, by \cite[ 13. 2. 8]{bsh}, $\Gamma_{R_+}(E)$ is also an *injective $R$-module. Moreover, by \ref{rem}(\ref{bh1.5.19}),  for any $i\in \NN$,
 \[0=  ^*\Ext^i_{R }(X, \Gamma_{R_+}(E))=  \Ext^i_{R }(X, \Gamma_{R_+}(E)).\]

Furthermore, we have the following homogeneous isomorphisms of graded $R$-modules
\begin{align}\label{a}
\Hom_R(X, \Gamma_{R_+}(M))&\cong \Hom_R(X,\underset{ \overset {\longrightarrow}{t\in \mathbb{N}}}{ \lim }\,\,   \Hom _{R}(R/ R_+^{t}, M) && \text{(by \cite[Theorem 1. 3. 8]{bsh})} \\
&\cong \underset{ \overset {\longrightarrow}{t\in \mathbb{N}}}{ \lim }\,\,\Hom_R(X, \Hom _{R}(R/ R_+^{t}, M)) && \text{(by \ref{rem}(\ref{limm})) } \nonumber\\
&\cong \underset{ \overset {\longrightarrow}{t\in \mathbb{N}}}{ \lim }\,\, \Hom_R(\tfrac{X}{R_+^{t}X}, M)  \nonumber\\
&=  \Gamma_{R_+}(X, M).\nonumber
\end{align}
The above isomorphisms are functorial, i.e.
\begin{equation}\label{wwww}
\Hom_R(-, \Gamma_{R_+}(-))\cong  \Gamma_{R_+}(-, -):
 ^*\mathfrak{C}^f_R\times  ^*\mathfrak{C}^f_R\rightarrow
  ^*\mathfrak{C}_R.
\end{equation}
Now, Lemma \ref{rot10.47} completes the proof.
\item
 
We have the following homogeneous isomorphisms of graded $R$-modules
\begin{align}\label{aa}
\Gamma_{R_+}(\Hom_R(X, M))&\cong \underset{ \overset {\longrightarrow}{t\in \mathbb{N}}}{ \lim }\,\, \Hom_R(\tfrac{R}{R_+^{t}}, \Hom_R(X, M)) \nonumber\\
&\cong \underset{ \overset {\longrightarrow}{t\in \mathbb{N}}}{ \lim }\,\, \Hom_R(\tfrac{X}{R_+^{t}X}, M) \nonumber\\
&\cong \underset{ \overset {\longrightarrow}{t\in \mathbb{N}}}{ \lim }\,\, \Hom_R(X, \Hom_R(\tfrac{R}{R_+^{t}}, M)) \nonumber\\
&\cong \Hom_R(X, \underset{ \overset {\longrightarrow}{t\in \mathbb{N}}}{ \lim }\,\, \Hom_R(\tfrac{R}{R_+^{t}}, M)) \nonumber &&\text{(by  \ref{rem}(\ref{limm}))}  \\
&\cong \Hom_R(X, \Gamma_{R_+}(M)).
\end{align}
The above isomorphisms are functorial, i.e.
\begin{equation}\label{cc}
\Hom_R(-, \Gamma_{R_+}(-))\cong \Gamma_{R_+}(\Hom_R(-, -)):
 ^*\mathfrak{C}^f_R\times  ^*\mathfrak{C}^f_R\rightarrow
 ^*\mathfrak{C}_R.
\end{equation}
Now, let $F^X_{\bullet}$ be a finite graded free resolution of $X$, i.e. a homogeneous resolution consisting of finitely generated graded free $R$-modules.
Then, it is easy to see that, for any *injective $R$-module $E$, the complex $\Hom_R(F^X_{\bullet}, E)$ is an *injective resolution of $\Hom_R(X, E)$. Therefore,
\begin{align}
H^i_{R_+}(\Hom_{R }(X, E))&= H^i(\Gamma_{R_+}(\Hom_R(F^X_{\bullet}, E)))\nonumber\\
&\cong H^i(\Hom_R(F^X_{\bullet},  \Gamma_{R_+}(E))) &&\text{(by \ref{cc})}\nonumber\\
&= \Ext^i_{R }(X, \Gamma_{R_+}(E))\nonumber\\
&=  ^*\Ext^i_{R }(X, \Gamma_{R_+}(E))=0.  &&\text{(by \cite[13. 2. 8]{bsh}  and   \ref{rem}(\ref{bh1.5.19}))}
\end{align}
Now, (\ref{wwww}) and (\ref{cc}) imply the following isomorphism in $ ^*\mathfrak{C}^f_R$,
\[\Gamma_{R_+}(\Hom_R(X, -)) \cong \Gamma_{R_+}(X, -),\]
and  Lemma \ref{rot10.47}  completes the proof.
\end{enumerate}
\end{proof}

\section{Bass numbers of $H^{f_{R_+}(M}_{R_+}(M)_n$}
Let $S$ be a ring and $\fa$ be an ideal of $S$.
The $\fa$-finiteness dimension of an $S$-module $X$ is defined by
\begin{equation}\label{deff}
    f_{\fa}(X):= \inf\{i\in \NN_0\arrowvert H^i_{\fa}(X) \,\, \text{is not finitely generated}\}.
\end{equation}
  In general, $f_{\fa}(X)$ may be  infinite. Moreover, if $X$ is finitely generated, then by \cite[Proposition 9. 1. 2]{bsh},
 \[f_{\fa}(X)= \inf\{i\in \NN\arrowvert \fa\nsubseteq \sqrt{0:_S H^i_{\fa}(X)}\}.\]
Graded components of local cohomology modules, at the level of   finiteness dimension, exhibit favorable asymptotic behavior. For example:
\begin{itemize}
   \item  (\cite[Theorem 3. 6]{jz}) Let $\fa_0$ be an ideal of $R_0$ and set $f:= f_{\fa_0+ R_+}(M)$. Then the set $\{\Ass_{R_0}(H^{f }_{\fa_0+ R_+}(M)_n)\}_{n\in \Z}$ is asymptotically stable; that is, there exists $t\in \Z$ such that
    \[\Ass_{R_0}(H^{f }_{\fa_0+ R_+}(M)_n)= \Ass_{R_0}(H^{f }_{\fa_0+ R_+}(M)_t),\,\,\text{for all}\,\,n\leq t.\]
   
\end{itemize}

In this section we study the asymptotic behavior of the Bass numbers of the graded components  
  $H^i_{R_+}(M)_n$ as $n\rightarrow -\infty $ when $i\leq f_{R_+}(M)$. 

Recall that, if $S$ is a ring, $X$ is an $S$-module and $\p\in \Spec(S)$, then for any $i\in \NN_0$ the $i$-th Bass number of $X$ with respect to $\p$ is defined by
\begin{equation}\label{Bass}
    \mu^i(\p, X):= \dim_{\kappa(p)}(\Ext^i_S(\tfrac{S}{\p}, X)_{\p}).
\end{equation}
Here, $\kappa({\p_0}):= \tfrac{(R_0)_{\p_0}}{\p_0(R_0)_{\p_0}},$ is the residue field of the local ring $(R_0)_{\p_0}$.
The really important property of the
$i$-th Bass number of $X$ with respect to $\p$  is that it gives, when the (uniquely determined) $i$-th term in the minimal injective resolution of $X$ is expressed as a direct sum of indecomposable injective $S$-modules, the number of those direct summands which
are isomorphic to $E_S(\tfrac{S}{\p})$, the injective envelope of the $S$-module $\tfrac{S}{\p}$.

The following   lemma will be used repeatedly, we include it here for the reader’s convenience.This  lemma provides a characterization of Noetherian and Artinian graded modules over standard graded algebras.
\begin{lem}\label{kirby}(\cite[Theorem 1]{kirby})
Let $N$ be a graded $R$-module. Then
\begin{enumerate}
    \item  $N$ is Noetherian if and only if there exist integers $t$ and $s$ such that
    \begin{enumerate}
        \item  $ N_n= 0$ for all $n\le t$,
            \item $N_{n+1}= R_1N_n$ for all $n\geq s$ and
                \item $N_n$ is a finitely generated $R_0$-module for all $t\leq n\leq s$.
 \end{enumerate}

     \item $N$ is Artinian if and only if there exist integers $u$ and $v$ such that
     \begin{enumerate}
\item  $ N_n= 0$  for all $n\geq u$,
\item $0:_{N_n} R_1= 0$ for all $n\leq v$ and
\item $N_n$ is an Artinian $R_0$-module for all $v\leq n\leq u$.
 \end{enumerate}

  \end{enumerate}
\end{lem}
In the following   theorem, which constitute one of the main results of this paper, we study the asymptotic behavior of the Bass numbers $\{\mu^i(\m_0, H^j_{R_+}(M)_n)\}_{n\in \Z}$, as $n\rightarrow -\infty$, where $(R_0,\m_0)$ is local. We show that these Bass numbers exhibit favorable properties in the case where
    $j\leq f_{R_+}(M)$.

\begin{thm}\label{thm1}
Assume that $(R_0,\m_0)$ is local and $R$ is a flat $R_0$-module via $R_0\hookrightarrow R$. Set $f:= f_{R_+}(M)$. Then,  for any  $l\leq f$, the following statements hold: 
\begin{enumerate}
\item 
 $\mu^{0}(\m_0, H^l_{R_+}(M)_n)=  \dim_{\tfrac{R_0}{\m_0}}(H^{l }_{R_+}(\tfrac{R}{\m_0R}, M)_n), \,\,\,\text{ for all}\,\, n\ll0.$
\item   $\mu^1(\m_0, H^l_{R_+}(M)_n)\leq \dim_{\tfrac{R_0}{\m_0}}(H^{l+ 1}_{R_+}(\tfrac{R}{\m_0R}, M)_n), \,\,\,\text{ for all}\,\, n\ll0.$
    \end{enumerate}
    \end{thm}
\begin{proof}
 Let  $l\leq f$.
 \begin{enumerate}
        \item Using Lemma \ref{spec}(\ref{spec1}), there exists the following homogeneous convergence of spectral sequences
\begin{equation}\label{s1}
E_2^{i, j}= \Ext^i_{R }(\tfrac{R}{\m_0R}, H^j_{R_+}(M))\underset{i}{\Rightarrow}H^{i+ j}_{R_+}(\tfrac{R}{\m_0R}, M).
\end{equation}
    
     Let $i\in \NN_0$ and $k< f$. Then
 \begin{align*}\label{9}
    (E_2^{i, k})_n&=  \Ext^{i}_{R}(\tfrac{R}{\m_0R},H^k_{R_+}(M))_n \nonumber\\
    &\cong \Ext^{i}_{R_0}(\frac{R_0}{\m_0},H^k_{R_+}(M)_n)  &&\text{ (by Lemma \ref{2})}\nonumber\\
   &= 0 \,\,\,\,\,\,\,\text{for all}\,\,n\ll0  &&\text{(by\,\,Lemma \ref{kirby})}.
  \end{align*}
Therefore, 
\[ (E_{\infty}^{i, k})_n= 0,\,\,\,\,\,\,\,\,\,\,\,\,\,\,\,\text{for all}\,\,i\in \NN_0, k<f \,\,\text{and all}\,\,n\ll0.\]
In particular,
\begin{equation}\label{99}
    \Hom_{R_0}(\tfrac{R_0}{\m_0},H^l_{R_+}(M)_n)= (E_2^{0, l})_n= (E_{\infty}^{0, l})_n.
\end{equation}
By convergence of the spectral sequence, for all $n\ll0$, there   exists a filtration
\[0\subseteq \Phi_{l }^n\subseteq\cdots\subseteq  \Phi_{1}^n\subseteq \Phi_{0}^n= H^{  l }_{R_+}(\tfrac{R}{\m_0R}, M)_n, \]
of submodules of $ H^{ l}_{R_+}(\tfrac{R}{\m_0R}, M)_n$ such that
\begin{align*}\label{66}
&(E_{\infty}^{0, l})_n\cong \tfrac{\Phi_{ 0}^n}{\Phi_{1}^n}, \,\, \text{and}\\
& 0= (E_{\infty}^{i, k})_n\cong   \tfrac{\Phi_{ i}^n}{\Phi_{i+ 1}^n}\,\,\text{for all }\,\,i,k \,\text{with}\,\,i+ k= l \,\,\text{and}\,\,k< l.
\end{align*}
Hence, the filtration reduces to
\[0= \Phi_{l }^n=\cdots=  \Phi_{1}^n\subseteq \Phi_{0}^n= H^{  j }_{R_+}(\tfrac{R}{\m_0R}, M)_n.\]

Therefore, by \ref{99}, for all $n\ll0$,
\begin{align*}
    \Hom_{R_0}(\tfrac{R_0}{\m_0}, H^{l}_{R_+}(M)_n)&\cong \Hom _{R}(\tfrac{R}{\m_0R}, H^{l}_{R_+}(M))_n\\
    &= \Phi_{ 0}^n= H^{ l }_{R_+}(\tfrac{R}{\m_0R}, M)_n,
\end{align*}
  and the result follows.
     \item Using \ref{s1} and Lemma \ref{kirby}, for all $i\in \NN_0$, $k< l$ and all $n\ll0$ we have
\begin{align*}\label{10}
(E_{2}^{i, k})_n&= \Ext^{i}_{R}(\tfrac{R}{\m_0R}, H^{ k}_{R_+}(M))_n\\
&\cong \Ext^{i}_{R_0}(\tfrac{R_0}{\m_0}, H^{ k}_{R_+}(M)_n)\\
&=0= (E_{\infty}^{i, k})_n,   
\end{align*}
and hence
\begin{align*}
    \Ext^{1}_{R}(\tfrac{R}{\m_0R}, H^{ l}_{R_+}(M))_n&= (E_2^{1, l})_n\\
    &\cong  (E_{\infty}^{1, l})_n, \,\,\,\,\,\,\,\,\text{for all}\,\,n\ll0.
\end{align*}
 Again, by convergence of the spectral sequence, for all $n\ll0$, there   exists a filtration
\[0\subseteq \Phi_{l+ 1}^n\subseteq\cdots\subseteq \Phi_{   2}^n\subseteq\Phi_{1}^n\subseteq \Phi_{0}^n= H^{  l+ 1}_{R_+}(\tfrac{R}{\m_0R}, M)_n, \]
of submodules of $ H^{ l+ 1}_{R_+}(\tfrac{R}{\m_0R}, M)_n$ such that
\begin{align*}\label{66}
&(E_{\infty}^{1, l})_n\cong \tfrac{\Phi_{ 1}^n}{\Phi_{2}^n}, \,\, \text{and}\\
& 0= (E_{\infty}^{i, k})_n\cong   \tfrac{\Phi_{ i}^n}{\Phi_{i+ 1}^n}\,\,\text{for all }\,\,i,k,\,\text{with}\,\,i+ k= l+ 1\,\,\text{and}\,\,k< l.
\end{align*}
Hence, \[(E_{\infty}^{1, l})_n\cong \Phi_{ 1}^n\subseteq  H^{ l+ 1}_{R_+}(\tfrac{R}{\m_0R}, M)_n.\]
Therefore, for all $n\ll0$,
\begin{align*}
    \Ext^{ 1}_{R_0}(\tfrac{R_0}{\m_0}, H^{l}_{R_+}(M)_n)&\cong \Ext^{ 1}_{R}(\tfrac{R}{\m_0R}, H^{l}_{R_+}(M))_n\\
    &\cong \Phi_{ 1}^n\subseteq H^{ l+ 1}_{R_+}(\tfrac{R}{\m_0R}, M)_n.
\end{align*}
Consequently,
 \[\mu^{1}(\m_0, H^l_{R_+}(M)_{n})\leq \dim_{\frac{R_0}{\m_0}}(H^{ l+ 1}_{R_+}(\tfrac{R}{\m_0R}, M)_n), \,\,\,\,\,\,\,\,\text{for all}\,\,n\ll0.\]
\end{enumerate}
  \end{proof}  
Note that $f_{R_+}(M)\geq 1$. Therefore, the following corollary is an immediate consequence of the preceding theorem.
\begin{cor}
   Let $(R_0,\m_0)$ be local and $R$ is a flat $R_0$-module via $R_0\hookrightarrow R$. Then,   
\begin{enumerate}
\item 
 $\mu^{0}(\m_0, H^1_{R_+}(M)_n)=   \dim_{\tfrac{R_0}{\m_0}}(H^{1}_{R_+}(\tfrac{R}{\m_0R}, M)_n), \,\,\,\text{ for all}\,\, n\ll0.$
\item   $\mu^1(\m_0, H^1_{R_+}(M)_n)\leq \dim_{\tfrac{R_0}{\m_0}}(H^{2}_{R_+}(\tfrac{R}{\m_0R}, M)_n),\,\,\,\text{ for all}\,\, n\ll0.$
    \end{enumerate}
    
\end{cor}

 \section{Bass numbers of $H^{\cd_{R_+}(M)}_{R_+}(M)_n$} 
Let $S$ be a ring. The  cohomological dimension
  of an $S$-module $X$ with respect to an ideal $\fa$ is defined by
\begin{equation}\label{defcd}
    \cd_{\fa}(X):= \sup\{i\in \NN_0\arrowvert H^i_{\fa}(X)\neq 0\}.
\end{equation}
Note that, in view of \cite[Corollary 3. 3. 3]{bsh}, if  $S$ is Noetherian, then $\cd_{\fa}(X)< \infty.$

The cohomological dimension is an important invariant  in local cohomology theory and has attracted considerable interest; see, for example, \cite{dnt, l} and \cite{o}.

Graded components of local cohomology modules, at the level of cohomological dimension  exhibit favorable asymptotic behavior. For example:
\begin{itemize}
  \item   (\cite[Corollary 3. 7]{b}) The set $\{\Supp_{R_0}(H^{\cd_{R_+}(M)}_{R_+}(M)_n)\}_{n\in \Z}$ is  decreasing; that is, 
         \[\Supp_{R_0}(H^{\cd_{R_+}(M)}_{R_+}(M)_{n- 1})\supseteq \Supp_{R_0}(H^{\cd_{R_+}(M)}_{R_+}(M)_n),\,\,\text{for all}\,\,n\in \Z.\]
     Moreover,  there exists $t\in \Z$ such that,  for all $n\leq t$,
     \[\Supp_{R_0}(H^{\cd_{R_+}(M)}_{R_+}(M)_n)= \{\p_0\in \Spec(R_0)| \dim_R(\kappa({\p_0})\otimes_{R_0} M)= \cd_{R_+}(M)\}.\]
     In other words, for all sufficiently small integers $n$,  the set $\Supp_{R_0}(H^{\cd_{R_+}(M)}_{R_+}(M)_n)$ depends only on $\Supp_R(M)$.
\end{itemize}
In this section we study the asymptotic behavior of the Bass numbers of the graded components  
  $H^i_{R_+}(M)_n$ as $n\rightarrow -\infty $ when $i=\cd_{R_+}(M)$. 
        
     The following lemma from \cite{dnt} will be used in the subsequent  theorem.
\begin{lem}\label{dnt}(\cite[Theorem 1.2]{dnt})
    Let $S$ be a ring, $\fa$ be an ideal of $S$ and $X$ and $Y$ be two finitely generated $S$-modules with
    $\Supp_S X \subseteq \Supp_S Y$. Then $\cd_{\fa}(X)\leq  \cd_{\fa}(Y)$. In particular, $\cd_{\fa}(X)= \cd_{\fa}(Y)$ whenever $\Supp_S X= \Supp_S Y$.
\end{lem}

In the following   theorem, which constitute one of the main results of this paper, we study the asymptotic behavior of the Bass numbers $\{\mu^i(\m_0, H^j_{R_+}(M)_n)\}_{n\in \Z}$, as $n\rightarrow -\infty$, where $(R_0,\m_0)$ is a regular local ring of dimension $d$. We show that these Bass numbers exhibit favorable properties in the case where
    $i\in \{d, d- 1\}$ and $j= \cd_{R_+}(M)$.

\begin{thm}\label{thm}
 
Assume that $(R_0,\m_0)$ is a regular local ring of dimension $d$ and $R$ is flat as an $R_0$-module
via $R_0\hookrightarrow R$. Set $c:= \cd_{R_+}(M)$. Then the following statements hold:
\begin{enumerate}

\item There exists a polynomial $p(x)$ of degree $< c$ such that
\[\mu^d(\m_0, H^c_{R_+}(M)_n)= p(n) \,\,\, \text{for  all} \,\, n\ll0.\]
In particular, if $c= 1$ then the sequence $\{\mu^d(\m_0,
H^1_{R_+}(M)_n)\}_{n\in \Z}$ of integers  is eventually constant; that is, there exists $\alpha\in \Z$  such that
 \[ \mu^d(\m_0, H^1_{R_+}(M)_n)=\alpha, \,\,\,\,\,\,\,\,\   \text{for all}\,\,n\ll0.\]

\item Suppose that $c> 0$ and   $\mu^d(\m_0, H^c_{R_+}(M))\neq 0$, then
$\mu^d(\m_0, M)\neq 0$ and the sequence $\{\mu^d(\m_0, H^c_{R_+}(M)_n)\}_{n\in \Z}$ is asymptotically decreasing. Moreover, if $c\geq 2$ it is asymptotically strictly decreasing.

\item If $H^c_{R_+}(\Ext^{d- 1}_{R }(\tfrac{R}{\m_0R}, M))$ is finitely generated then there exists a polynomial $p(x)\in \Q[x]$ of degree $< c-1$ such that
\[ \mu^{d- 1}(\m_0, H^c_{R_+}(M)_n)\leq  p(n), \,\,\, \text{for all}\,\, n\ll0.\]
If, in addition, $\Ext^{d}_{R }(\frac{R}{\m_0R}, H^{c- 1}_{R_+}(M))$ is finitely generated, then the equality holds.



\end{enumerate}
\end{thm}
\begin{proof}

  The homomorphism $R_0\hookrightarrow R$ is   faithfully flat. Since $R_0$ is regular, we have
$$\pd_R(\tfrac{R}{\m_0R})= \pd_{R_0}(\tfrac{R_0}{\m_0})= \dim(R_0)= d.$$
\begin{enumerate}
\item Consider the   homogeneous convergence of spectral sequences
\begin{equation*} 
E_2^{i, j}= \Ext^i_{R }(\tfrac{R}{\m_0R}, H^j_{R_+}(M))\underset{i}{\Rightarrow}H^{i+ j}_{R_+}(\tfrac{R}{\m_0R}, M) \ \ \ \ \ (\ref{s1}).
\end{equation*}
Since \[\Ext^i_{R }(\tfrac{R}{\m_0R}, H^j_{R_+}(M))= 0, \,\,\text{for }\,\,i> d\,\,\text{ or}\,\,j> c,\] it follows that
\[E_{\infty}^{i, j}= E_2^{i, j}= 0, \,\,\, \text{ for} \,\, i> d\,\,  \text{or} \,\, j> c.\]
Hence, by convergence of the spectral sequence,
\begin{equation}\label{b}
 \Ext^d_{R }(\tfrac{R}{\m_0R}, H^c_{R_+}(M))= E_2^{d, c}= E_{\infty}^{d, c}\cong H^{d+ c}_{R_+}(\tfrac{R}{\m_0R}, M).
\end{equation}
On the other hand, by Lemma \ref{spec}(\ref{spec2}), there is a second homogeneous convergent spectral sequence:
\begin{equation}\label{s2}
E_2^{i, j}= H^i_{R_+}(\Ext^j_{R }(\tfrac{R}{\m_0R},M))\underset{i}{\Rightarrow}H^{i+ j}_{R_+}(\tfrac{R}{\m_0R}, M).
\end{equation}
By Lemma \ref{dnt}, $\cd_{R_+}(\Ext^j_{R }(\tfrac{R}{\m_0R},M))\leq c$, and hence  \[H^i_{R_+}(\Ext^j_{R }(\tfrac{R}{\m_0R},M))= 0, \,\,\text{for}\,\,i> c \,\,\text{or}\,\,j>d. \]  Thus
\begin{equation}\label{d}
    E_{\infty}^{i, j}= E_2^{i, j}= 0, \,\,\,  \text{for} \,\, i> c\,\,  \text{or} \,\, j> d,
\end{equation}
and therefore,
\begin{equation}\label{bb}
H^{d+ c}_{R_+}(\tfrac{R}{\m_0R}, M)\cong E_{\infty}^{d, c}= E_2^{d, c}=  H^c_{R_+}(\Ext^d_{R }(\tfrac{R}{\m_0R},M)).
\end{equation}
Combining \ref{b}, \ref{bb} and Lemma \ref{2}, we obtain, for any $n\in \Z$,
\begin{align}\label{isomorph}
    H^c_{R_+}(\Ext^d_{R }(\tfrac{R}{\m_0R},M))_n&\cong \Ext^d_{R }(\tfrac{R}{\m_0R}, H^c_{R_+}(M))_n \nonumber\\
    &\cong \Ext^d_{R_0 }(\tfrac{R_0}{\m_0}, H^c_{R_+}(M)_n).
\end{align}
Let $\mu(X)$ denote the minimum number of generators of an
$R_0$-module $X$. It is well-known that $\mu(X)= \dim_{\tfrac{R_0}{\m_0}}(\tfrac{X}{\m_0X})$. Thus, \ref{isomorph} yields
\begin{align}\label{m}
    \mu^d(\m_0, H^c_{R_+}(M)_n)&=
\dim_{\tfrac{R_0}{\m_0}}(H^c_{R_+}(\Ext^d_{R
}(\tfrac{R}{\m_0R},M))_n)  \nonumber\\
&= \mu(H^c_{R_+}(\Ext^d_{R }(\tfrac{R}{\m_0R},M))_n).
\end{align}
By Lemma \ref{dnt}, $\cd_{R_+}(\Ext^d_{R }(\tfrac{R}{\m_0R}, M))\leq  \cd_{R_+}(M)= c.$

 If $\cd_{R_+}(\Ext^d_{R }(\tfrac{R}{\m_0R},M))< c,$ Then
\[0= \mu(H^c_{R_+}(\Ext^d_{R }(\tfrac{R}{\m_0R}, M))_n)= \mu^d(\m_0, H^c_{R_+}(M)_n), \,\,\text{for all}\, n\in \Z.\]
Otherwise, if equality holds, then by \cite[Corollary 2. 4]{b}, there exists a polynomial
$p(x)\in \Q[x]$ of degree $< c$ such that
\[\mu^d(\m_0, H^c_{R_+}(M)_n)= p(n), \,\,\, \text{for  all}\,\, n\ll0.\]
In particular, if $c= 1$ then $p(x)= k\in \Q$, is constant, and hence the sequence $\{\mu^d(\m_0,
H^1_{R_+}(M)_n)\}_{n\in \Z}$ is eventually constant.

\item  Assume that $c> 0$ and  $\mu^d(\m_0, H^c_{R_+}(M))\neq 0$. Since
\begin{align}
\mu^d(\m_0, H^c_{R_+}(M))&= \dim_{\tfrac{R_0}{\m_0}}(\Ext^d_{R_0}(\tfrac{R_0}{\m_0}, H^c_{R_+}(M))) \nonumber\\
&= \sum_{ n\in \Z}\dim_{\tfrac{R_0}{\m_0}}(\Ext^d_{R_0}(\tfrac{R_0}{\m_0}, H^c_{R_+}(M)_n)),   &&\text{ (by Lemma \ref{rem}(\ref{limm}))} \nonumber
\end{align}
 there exists $t\in \Z$ such that
$\mu^d(\m_0, H^c_{R_+}(M)_t)\neq 0$. By \ref{isomorph}, this implies that $\Ext^d_R(\frac{R}{\m_0R},M)\neq 0$ and hence $\mu^d(\m_0, M)\neq 0$.  Again by \ref{isomorph}, we obtain
\[H^c_{R_+}(Ext^d_{R
}(\tfrac{R}{\m_0R},M))_t\neq 0.\]
 Therefore, in view of Lemma \ref{dnt}, $\cd_{R_+}(\Ext^d_{R }(\tfrac{R}{\m_0R},M))= c> 0$.

Now, by  \cite[Corollary 2.4]{b} and \ref{isomorph}, we have
\begin{align}\label{3}
  \mu^d(\m_0, H^c_{R_+}(M)_{t- 1}) & = \mu(H^c_{R_+}(\Ext^d_{R
}(\frac{R}{\m_0R},M))_{t- 1}) \nonumber\\
&\geq \mu(H^c_{R_+}(\Ext^d_{R
}(\tfrac{R}{\m_0R},M))_{t})+ c-1  \nonumber \\
&= \mu^d(\m_0, H^c_{R_+}(M)_{t})+ c- 1    \nonumber \\
&\geq  \mu^d(\m_0, H^c_{R_+}(M)_{t})> 0.
 \end{align}
Therefore, the sequence $\{\mu^d(\m_0, H^c_{R_+}(M)_n)\}_{n\in \Z}$ is   decreasing for all $n\leq t$. Moreover, if $c\geq 2$, then  it is   strictly decreasing.

\item Assume that $H^c_{R_+}(Ext^{d- 1}_{R
}(\tfrac{R}{\m_0R},M))$ is finitely generated. Then, by Lemma \ref{kirby},
\begin{equation}\label{4}
  H^c_{R_+}(\Ext^{d- 1}_{R}(\tfrac{R}{\m_0R},M))_n =0, \,\,\,\,\,\,\,\,\,\, \text{ for all}\,\, n\ll0.
\end{equation}
From  the convergence of spectral sequences \ref{s2}, we obtain
\[ H^{c- 1}_{R_+}(\Ext^{d}_{R}(\tfrac{R}{\m_0R},M))_n =(E_2^{c- 1, d})_n\cong (E_{\infty}^{c- 1, d})_n, \,\,\,\,\,\,\,\,\,\text{for all}\,\, n\ll0.\]
Hence, for all $n\ll0$, there exists a filtration
\[0\subseteq \Phi_{d+ c- 1}^n\subseteq \cdots\subseteq\Phi_{ c+ 1}^n\subseteq\Phi_{ c}^n\subseteq\Phi_{c- 1}^n\subseteq\cdots \subseteq\Phi_{0}^n= H^{d+ c- 1}_{R_+}(\tfrac{R}{\m_0R}, M)_n, \]
of submodules of $ H^{d+ c- 1}_{R_+}(\frac{R}{\m_0R}, M)_n$ such that
\begin{equation}\label{5}
  H^{c- 1}_{R_+}(\Ext^{d}_{R}(\tfrac{R}{\m_0R},M))_n\cong \tfrac{\Phi^n_{c- 1}}{\Phi^n_{c}}.
\end{equation}
By the vanishing of all other $E_{\infty}^{i, j}$-terms with  $i+ j= d+ c- 1$ and $i\neq c- 1$, the filtration collapses, for all $n\ll 0$, and we obtain
\begin{equation}\label{6}
  H^{c- 1}_{R_+}(\Ext^{d}_{R}(\tfrac{R}{\m_0R},M))_n\cong H^{d + c- 1}_{R_+}(\tfrac{R}{\m_0R}, M)_n,\,\,\,\,\,\,\,\,\,\,\,\,\text{ for all}\,\, n\ll0.
\end{equation}
Now, consider the convergence of spectral sequences \ref{s1}. Since
\[E_2^{i, j}= \Ext^{i}_{R}(\tfrac{R}{\m_0R}, H^{j}_{R_+}(M))=0, \,\,\text{for}\,\, i>d \,\,\text{or} \,\,\ j>c,\]
we have
\[ \Ext^{d- 1}_{R}(\tfrac{R}{\m_0R}, H^{c}_{R_+}(M))= E_2^{d- 1, c}= E_{\infty}^{d- 1, c}.\]
Hence,   there exists a filtration  
\[0\subseteq \Phi_{d+ c- 1}\subseteq \cdots\subseteq\Phi_{ d+ 1}\subseteq\Phi_{d}\subseteq\Phi_{d- 1}\subseteq\cdots \subseteq\Phi_{0}= H^{d+ c- 1}_{R_+}(\tfrac{R}{\m_0R}, M), \]
of submodules of $ H^{d+ c- 1}_{R_+}(\tfrac{R}{\m_0R}, M)$ such that
\begin{equation}\label{66}
 \Ext^{d- 1}_{R}(\tfrac{R}{\m_0R}, H^{c}_{R_+}(M))\cong \tfrac{\Phi_{d- 1}}{\Phi_{d}}.
\end{equation}
Since for all  $0\leq i, j\leq d+ c- 1$   with $i+ j= d+ c- 1$  and $i\neq d- 1, d$,
\[0= \Ext^{i}_{R}(\tfrac{R}{\m_0R}, H^{j}_{R_+}(M))= E_2^{i, j}=  E_{\infty}^{i, j}= \tfrac{\Phi_{i}}{\Phi_{i+ 1}},\]
the above filtration reduces to
\[0= \Phi_{d+ c- 1}= \cdots=\Phi_{ d+ 1}\subseteq\Phi_{ d}\subseteq\Phi_{d- 1}=\cdots =\Phi_{0}= H^{d+ c- 1}_{R_+}(\tfrac{R}{\m_0R}, M).\]
Therefore, \ref{66} implies
\begin{equation}\label{7}
  \Ext^{d- 1}_{R}(\tfrac{R}{\m_0R}, H^{c}_{R_+}(M))\cong  \tfrac{H^{d+ c- 1}_{R_+}(\tfrac{R}{\m_0R}, M)}{\Phi_{ d}}.
\end{equation}
Moreover, by \ref{6}, for all $n\ll 0$ we obtain
\begin{align}\label{8}
\mu^{d- 1}(\m_0, H^c_{R_+}(M)_{n})&\leq \dim_{\tfrac{R_0}{\m_0}}(H^{d+ c- 1}_{R_+}(\tfrac{R}{\m_0R}, M)_n) \nonumber \\
&= \dim_{\tfrac{R_0}{\m_0}}(H^{c- 1}_{R_+}(\Ext^{d}_{R}(\tfrac{R}{\m_0R},M))_n).
\end{align}
Now, consider the standard graded ring ${R^\prime}:= \tfrac{R_0}{\m_0}\bigotimes_{R_0}R$ with base ring $\tfrac{R_0}{\m_0}$. By the Independence Theorem \cite[Theorem 14. 1. 7]{bsh}, there is a homogeneous isomorphism
\[H^{c- 1}_{R_+}(\Ext^{d}_{R}(\tfrac{R}{\m_0R},M))\cong H^{c- 1}_{R^{\prime}_{+}}(\Ext^{d}_{R}(\tfrac{R}{\m_0R},M)).\]
Hence, by \cite[Theorem 17.1.11]{bsh}, there exists a polynomial $p(x)\in \Q[x]$ of degree $< c- 1$ such that
\[\dim_{\frac{R_0}{\m_0}}(H^{c- 1}_{R_+}(\Ext^{d}_{R}(\tfrac{R}{\m_0R},M))_n)= p(n),\,\,\,\,\,\,\, \text{for all}\,\,n\ll0.\]
Therefore, \ref{8} implies that
\[\mu^{d- 1}(\m_0, H^c_{R_+}(M)_{n})\leq p(n),\,\,\,\,\,\,\, \text{for all}\,\,n\ll0,\]
as desired.

Finally, assume in addition that $\Ext^{d}_{R }(\tfrac{R}{\m_0R}, H^{c- 1}_{R_+}(M))$ is a finitely generated $R$-module. Then, by lemmas \ref{kirby} and \ref{2},
\[  \Ext^{d}_{R_0 }(\tfrac{R_0}{\m_0}, H^{c- 1}_{R_+}(M)_n)= 0, \,\,\,\,\,\text{for all }\,\,n\ll0.\]
Using again the spectral sequences \ref{s1} we obtain
\begin{align}
\Ext^{d- 1}_{R_0}(\tfrac{R_0}{\m_0}, H^{c}_{R_+}(M)_n)= (E_2^{d- 1, c})_n&= (E_{\infty}^{d- 1, c})_n \nonumber \\
&\cong H^{d+ c- 1}_{R_+}(\tfrac{R}{\m_0R}, M)_n.
\end{align}
Therefore, the only inequality in \ref{8} is in fact an  equality and the result follows using the same argument as used above.

\end{enumerate}
\end{proof}
Note that, in view of Lemma \ref{dnt}, for all $i\in \NN_0$,
\[\cd_{R_+}(\Ext^{i}_{R }(\tfrac{R}{\m_0R}, M))\leq \cd_{R_+}(M).\]
So, the condition that $H^c_{R_+}(\Ext^{d- 1}_{R }(\tfrac{R}{\m_0R}, M))$ is finitely generated in the item 3 of the above theorem is not very restrictive.

The following corollary is immediate by the above theorems.
\begin{cor}
  Assume that  $R$ is flat as an $R_0$-module
via $R_0\hookrightarrow R$, let $\p_0\in \Spec(R_0)$ and  set $c:= \cd_{R_+}(M)$  and $f:= f_{R_+}(M)$. Then the following statements hold.
\begin{enumerate}
\item If $l\leq f$, then  
\[ \mu^{0}(\p_0, H^l_{R_+}(M)_n)= \dim_{\kappa{(\p_0)}}((H^{l }_{R_+}(\tfrac{R}{\p_0R}, M)_n)_{\p_0}),\,\,\,\text{ for all}\,\, n\ll0.\]
Moreover, 
 \[\mu^1(\p_0, H^l_{R_+}(M)_n)\leq \dim_{\kappa{(\p_0)}}((H^{l+ 1}_{R_+}(\tfrac{R}{\p_0R}, M)_n)_{\p_0}),\,\,\,\text{ for all}\,\, n\ll0.\]
\item Assume that $R_0$ is regular  and set $p:= \hei(\p_0)$. Then
\begin{enumerate}
 \item There exists a polynomial $p(x)$ of degree $< c$ such that
\[\mu^p(\p_0, H^c_{R_+}(M)_n)= p(n) \,\,\, \text{for  all} \,\, n\ll0.\]
In particular, if $c= 1$ then the sequence $\{\mu^p(\p_0,
H^1_{R_+}(M)_n)\}_{n\in \Z}$  is eventually constant.

\item    If 
 $c> 0$ and   $\mu^p(\p_0, H^c_{R_+}(M))\neq 0$, then
$\mu^p(\p_0, M)\neq 0$ and the sequence $\{\mu^p(\p_0, H^c_{R_+}(M)_n)\}_{n\in \Z}$ is asymptotically decreasing.
 Moreover, if $c\geq 2$ it is asymptotically strictly decreasing.

\item If $H^c_{R_+}(Ext^{p- 1}_{R }(\frac{R}{\p_0R}, M))$ is finitely generated, then there exists a polynomial $p(x)\in \Q[x]$ of degree $< c-1$ such that
\[ \mu^{p- 1}(\p_0, H^c_{R_+}(M)_n)\leq  p(n) \,\,\, \text{for  all}\,\, n\ll0.\]
If, in addition, $Ext^{p}_{R }(\frac{R}{\p_0R}, H^{c- 1}_{R_+}(M))$ is finitely generated then the equality holds.


   
\end{enumerate}
\end{enumerate}
\end{cor}
\begin{proof}
    Consider the flat extension $R^{\prime}:= (R_{0})_{\p_0}\otimes_{R_{0}}R$ of $R$ and let $M^{\prime}:= M\otimes_R R^{\prime}$. Then $R^{\prime}$ is a standard graded ring with local base ring $((R_{0})_{\p_0}, \p_0(R_{0})_{\p_0})$.
    
    Moreover, in view of the graded Flat Base Change Theorem \cite[Theorem 14. 1. 9]{bsh}, for any $i\in \NN_0$  we have
    \begin{align*}
      (H^i_{R_+}(M)_n)_{\p_0} &\cong  (H^i_{R_+}(M)_{\p_0})_n \\
      &\cong  (H^i_{R_+}(M)\otimes_{R_0} (R_{0})_{\p_0})_n \\
      &\cong (H^i_{R_+}(M)\otimes_{R } R^{\prime} )_n\\
      &\cong H^i_{R^{\prime}_+}(M^{\prime})_n.  
    \end{align*}
    Therefore, for any $i,j\in \NN_0$ we obtain
    \begin{align*}
        \mu^{i}(\p_0, H^j_{R_+}(M)_{n})&= \dim_{\kappa(\p_0)}(\Ext^i_{(R_0)_{\p_0}}(\kappa(\p_0), (H^j_{R_+}(M)_n)_{\p_0}))\\
        &= \dim_{{\frac{R^{\prime}_0}{\p_0R^{\prime}_0}}}(\Ext^i_{R^{\prime}_0}(\tfrac{R^{\prime}_0}{\p_0R^{\prime}_0}, H^j_{R^{\prime}_+}(M^{\prime})_n)).
    \end{align*}
    
    In addition, it is straightforward to verify that
    $\cd_{R^{\prime}_+}(M^{\prime})\leq \cd_{R_+}(M)$, $f_{R_+}(M)\leq f_{R^{\prime}_+}(M^{\prime})$ and that $\pd_{R^\prime}(\tfrac{R^{\prime}}{\p_0R^{\prime}})\leq \pd_R(\tfrac{R}{\p_0R})$.

    Finally,  if $R$ is flat over $R_0$, then $R^{\prime}$ is flat as an $R^{\prime}_0$-module via $R^{\prime}_0\hookrightarrow R^{\prime}$.
    Now, the results follow from the Theorems \ref{thm1} and \ref{thm}.
\end{proof}

\begin{rem}
    \begin{enumerate}
        \item In view of \cite[Lemma 3. 8]{b}, for any $i\in \NN_0$,
        \[\Ass_R(H^i_{R_+}(M))= \bigcup_{n\in \Z}\{\p_0+ R_+\arrowvert \p_0\in \Ass_{R_0}(H^i_{R_+}(M)_n)\}.\]
        Therefore,
        \[ E_R(H^i_{R_+}(M))=\oplus_{n\in \Z}\oplus_{\p_0\in \Ass_{R_0}(H^i_{R_+}(M)_n)}E_R(\tfrac{R}{\p_0+ R_+})^{\mu^0(\p_0+ R_+, H^i_{R_+}(M))}.\]
        Also, for all $n\in \Z,$
        \[E_{R_0}(H^i_{R_+}(M)_n)=  \oplus_{\p_0\in \Ass_{R_0}(H^i_{R_+}(M)_n)}  E_{R_0}(\tfrac{R_0}{\p_0})^{\mu^0(\p_0, H^i_{R_+}(M)_n)}.\]
        In other words,
       \begin{equation*}
           \mu^0(\p_0+ R_+, H^i_{R_+}(M))> 0  
            \Leftrightarrow\exists n\in \Z \,\,\text{such that }\,\,\mu^0(\p_0, H^i_{R_+}(M)_n)> 0.
       \end{equation*}
       \item
       For any $i, j \in \NN_0$ and $\p_0\in \Spec(R_0)$ we have
       \begin{align}
            \mu^i(\p_0, H^j_{R_+}(M))&= \dim_{\kappa(\p_0)}(\Ext^i_{(R_0)_{\p_0}}(\kappa(\p_0), H^j_{R_+}(M)_{\p_0}))\nonumber\\
       & =\sum _{n\in \Z}\dim_{\kappa(\p_0)}(\Ext^i_{(R_0)_{\p_0}}(\kappa(\p_0), (H^j_{R_+}(M)_n)_{\p_0}))\nonumber&&\text{(by \ref{rem}(\ref{limm}))}\\
       & =  \sum_{n\in \Z}  \mu^i(\p_0, H^j_{R_+}(M)_n).
       \end{align}
       Therefore, by \cite[Theorem 16. 1. 5]{bsh},
       \begin{equation}\label{finite}
           \mu^i(\p_0, H^j_{R_+}(M))< \infty\Longleftrightarrow \mu^i(\p_0, H^j_{R_+}(M)_n)=0 \,\,\text{for all }\,\,n\ll 0.
       \end{equation} 
    \end{enumerate}
\end{rem}
We recall from \cite{jr} that if $S$ is a ring and $\fa$ is an ideal of $S$ then a non-zero finitely generated $S$-module $X$ is said to be relative Cohen-Macaulay with respect to $\fa$ of degree $c$ if 
\begin{equation}\label{defrcm}
    \grade_{\fa}(X)= \cd_{\fa}(X)=c,
\end{equation}  
 where $\grade_{\fa}(X)$ denotes the length of a maximal $X$-regular sequence in $\fa$. Equivalently,  $X$ is said to be relative Cohen-Macaulay with respect to $\fa$ of degree $c$ if $H^i_{\fa}(X)= 0$ for any $i\neq c$ (see \cite[Theorem 6. 2. 7]{bsh}).

The following corollary, consider the Bass numbers of local cohomology modules of relative Cohen-Macaulay modules.

\begin{cor}\label{rcm}
    Assume that  $(R_0, \m_0)$ is local,  $M$ is relative Cohen-Macaulay with respect to $R_+$ of degree $c$ and $R$ is flat as an $R_0$-module
via $R_0\hookrightarrow R$. Then   the following statements holds.
\begin{enumerate}
    \item For any $i\in \NN_0$ and  $n\in \Z$, $\mu^i(\m_0, H^c_{R_+}(M)_n)=  \dim_{\tfrac{R_0}{\m_0}}(H^{i+ c }_{R_+}(\tfrac{R}{\m_0R}, M)_n).$
 \item Let  $(R_0, \m_0)$ be regular of dimension $d$. Then,
 \begin{enumerate}
     \item There exists a polynomial $p(x)\in \Q[x]$ of degree $< c- 1$ such that 
     \[\mu^{d- 1}(\m_0, H^c_{R_+}(M)_n)\geq p(n),\,\,\,\,\,\,\,\,\text{for all}\,\,n\ll 0.\]
     \item If $H^{c- 1}_{R_+}(\Ext^d_R(\tfrac{R}{\m_0R}, M))$ is not finitely generated, then 
     \[\mu^{d- 1}(\m_0, H^c_{R_+}(M)_n)> 0,\,\,\,\,\,\,\,\,\text{for all}\,\,n\ll 0.\]
     Therefore, $\mu^{d- 1}(\m_0, H^c_{R_+}(M))$ is not finite.
 \end{enumerate}
\end{enumerate}
\end{cor}
\begin{proof}
    \begin{enumerate}
        \item Consider the following homogeneous convergence of spectral sequences
\[
E_2^{i, j}= \Ext^i_{R }(\tfrac{R}{\m_0R}, H^j_{R_+}(M))\underset{i}{\Rightarrow}H^{i+ j}_{R_+}(\tfrac{R}{\m_0R}, M).\,\,\,\,\,\, (\ref{s1})\]
   
    Since $M$ is relative Cohen-Macaulay with respect to $R_+$ of degree $c$, we have 
    \[E_2^{i, j}=  \Ext^{i}_{R}(\tfrac{R}{\m_0R}, H^j_{R_+}(M))= 0,\,\,\,\,\,\,\text{for all}\,\,i, j\in \NN_0 \,\,\text{with} \,\,j\neq c.\]
    Therefore, using the concept of convergence of spectral sequences, we have the following homogeneous isomorphisms of graded $R$- modules
    \begin{align}\label{c1}
       \Ext^{i}_{R}(\tfrac{R}{\m_0R}, H^c_{R_+}(M))&= E_2^{i, c}\nonumber\\
       &\cong E_{\infty}^{i, c}\nonumber\\
       &\cong H^{i+ c }_{R_+}(\tfrac{R}{\m_0R}, M),&&\text{for all}\,\,i\in \NN_0.  
    \end{align}
     This proves $(1)$.

    \item Assume that $R_0$ is regular of dimension $d$. Then, $$  \pd_{R}(\tfrac{R}{\m_0R})= \pd_{R_0}(\tfrac{R_0}{\m_0})= d.$$ 
    \begin{enumerate}
   \item  Considering the convergence of spectral sequences \ref{s2},
 \[E_2^{i, j}= H^i_{R_+}(\Ext^j_{R }(\tfrac{R}{\m_0R},M))\underset{i}{\Rightarrow}H^{i+ j}_{R_+}(\tfrac{R}{\m_0R}, M),\]
 we have
 \[E_{\infty}^{i, j}= E_2^{i, j}= 0, \,\,\,  \text{when} \,\, i> c\,\,  \text{or} \,\, j> d.\]
 
Hence, by the convergence of spectral sequences,
 
    \[  H^{c- 1}_{R_+}(\Ext^{ d}_{R }(\tfrac{R}{\m_0R},M))= E_2^{c- 1, d}\cong  E_{\infty}^{c- 1, d}, \]
 and this module is a subquotient of $H^{d+ c- 1}_{R_+}(\tfrac{R}{\m_0R}, M)$.
Now, by \ref{c1}, for any $n\in \Z$,
 \begin{equation}\label{c2}
     \mu^{d- 1}(\m_0, H^c_{R_+}(M)_n)\geq \dim_{\tfrac{R_0}{\m_0}}(H^{c- 1}_{R_+}(\Ext^{ d}_{R }(\tfrac{R}{\m_0R}, M))_n).
 \end{equation}
 Moreover, by \cite[Theorem 17. 1. 11]{bsh}, there is a polynomial $p(x)\in \Q[x]$ of degree $< c- 1$ such that 
     \[\dim_{\tfrac{R_0}{\m_0}}(H^{c- 1}_{R_+}(\Ext^{ d}_{R }(\tfrac{R}{\m_0R},M))_n)= p(n),\,\,\,\,\,\,\,\,\text{for all}\,\,n\ll 0.\]
    
 \item Assume that $H^{c- 1}_{R_+}(\Ext^d_R(\tfrac{R}{\m_0R}, M))$ is not finitely generated. Then, using \ref{kirby}, 
    \[H^{c- 1}_{R_+}(\Ext^d_R(\tfrac{R}{\m_0R}, M))_n \neq 0,\,\,\,\,\,\ \text{for all} \,\,n\ll 0. \]
   Hence, the result follows from \ref{c2} together with \ref{finite}.
     \end{enumerate}
     \end{enumerate}
\end{proof}
 

\end{document}